\documentclass[12pt,fleqn]{article}
\usepackage{amsmath,amsthm,amssymb,amscd}
\usepackage{graphicx}
\usepackage{graphics}
\usepackage{verbatim}
\usepackage{enumerate}
\usepackage{mathrsfs}
\usepackage{color}
\usepackage[english]{babel}
\makeatletter

\newcommand{\Rmnum}[1]{\expandafter\@slowromancap\romannumeral #1@}

\newtheorem{theorem}{Theorem}[section]

\newtheorem{problem}[theorem]{Problem}
\newtheorem{lemma}[theorem]{Lemma}

\newtheorem{corollary}[theorem]{Corollary}

\makeatother

\begin{document}
	
	\title{Plane graphs without 4- and 5-cycles and without ext-triangular 7-cycles are 3-colorable}
	
	\vspace{3cm}
	\author{Ligang Jin\footnotemark[1], Yingli Kang\footnotemark[1] \footnotemark[2], Michael Schubert\footnotemark[1] \footnotemark[2], Yingqian Wang\footnotemark[3]}
	\footnotetext[1]{Institute of Mathematics and Paderborn Center for Advanced Studies,
		Paderborn University,
		33102 Paderborn,
		Germany; ligang@mail.upb.de (Ligang Jin), Yingli@mail.upb.de (Yingli Kang), mischub@upb.de (Michael Schubert)}
	\footnotetext[2]{Fellow of the International Graduate School "Dynamic Intelligent Systems"}
	\footnotetext[3]{Department of Mathematics, Zhejiang Normal University, 321004 Jinhua, China;
		 yqwang@zjnu.cn (Yingqian Wang)}
		
	\date{}
	
	\maketitle

\begin{abstract}
Listed as No. 53 among the one hundred famous unsolved problems in [J. A. Bondy, U. S. R. Murty, Graph Theory, Springer, Berlin, 2008] is Steinberg's conjecture, which states that every planar
graph without 4- and 5-cycles is 3-colorable. In this paper, we show that plane graphs without 4- and 5-cycles are 3-colorable if they have no ext-triangular 7-cycles. This implies that
(1) planar graphs without 4-, 5-, 7-cycles are 3-colorable, and
(2) planar graphs without 4-, 5-, 8-cycles are 3-colorable,
which cover a number of known results in the literature motivated by Steinberg's conjecture.
\end{abstract}


\section{Introduction}
In the field of 3-colorings of planar graphs, one of the most active topics is about a conjecture proposed by Steinberg in 1976: every planar graph without cycles of length 4 and 5 is 3-colorable.
There had been no progress on this conjecture for a long time, until Erd\"{o}s \cite{Steinberg1993211} suggested a relaxation of it:
does there exist a constant $k$ such that every planar graph without cycles of length from 4 to $k$ is 3-colorable?
Abbott and Zhou \cite{AbbottZhou1991203} confirmed that such $k$ exists and $k\leq 11$.
This result was later on improved to $k\leq 9$ by
Borodin \cite{Borodin1996183} and, independently, by Sanders and Zhao \cite{SandersZhao199591}, and to
$k\leq 7$ by Borodin, Glebov, Raspaud and Salavatipour \cite{BorodinEtc2005303}.
\begin{theorem}[\cite{BorodinEtc2005303}] \label{thm_4to7}
Planar graphs without cycles of length from 4 to 7 are 3-colorable.
\end{theorem}

We remark that Steinberg's conjecture was recently shown to be false in \cite{CohenEtc2016}, by constructing a counterexample to the conjecture. 
The question whether every planar graph without cycles of length from 3 to 5 is 3-colorable is still open.

A more general problem than Steinberg's Conjecture was formulated in \cite{LuEtc20094596, Jin_2016_469}:
\begin{problem} \label{pro_4i}
What is $\cal{A}$, a set of integers between 5 and 9, such that for $i\in \cal{A}$, every planar graph with cycles
of length neither 4 nor $i$ is 3-colorable?
\end{problem}
Thus, Steinberg's Conjecture states that $5\in \cal{A}$. Since so far no element of $\cal{A}$ has been confirmed,
it seems reasonable to consider a relaxation of Problem \ref{pro_4i} where more integers are forbidden to be the length of a cycle in planar graphs.
Due to a famous theorem of Gr\"{o}tzsch that planar graphs without triangles are 3-colorable,
triangles are always allowed in further sufficient conditions.
Several papers together contribute to the result below: 
\begin{theorem}
	For any three integers $i,j,k$ with $5\leq i<j<k\leq
	9$,	it holds true that planar graphs having no cycles of length $4,i,j,k$ are 3-colorable.
\end{theorem}
Later on, the sufficient conditions, concerning three integers forbidden to be the length of a cycle, were considered. The corresponding problem can be formulated as follows:
\begin{problem} \label{pro_4ij}
	What is $\cal{B}$, a set of pairs of integers
	$(i,j)$ with $5\leq i<j\leq 9$, such that planar graphs without
	cycles of length $4, i, j$ are 3-colorable?
\end{problem}
It has been proved by Borodin et al. \cite{BorodinEtc2009668} and independently by Xu \cite{Xu2009347} that 
every planar graph having neither 5- and 7-cycles nor adjacent 3-cycles is 3-colorable. Hence, $(5,7)\in \cal{B}$, which improves on Theorem \ref{thm_4to7}.
More elements of $B$ have been confirmed: $(6,8)\in \cal{B}$ by Wang and Chen \cite{WangChen20071552}, 
$(7,9)\in \cal{B}$ by Lu et al. \cite{LuEtc20094596}, and $(6,9)\in \cal{B}$ by Jin et al. \cite{Jin_2016_469}.
The result $(6,7)\in \cal{B}$ is implied in the following theorem, which reconfirms the results $(5,7)\in \cal{B}$ and $(6,8)\in \cal{B}$.

\begin{theorem}[\cite{BorodinEtc20102584}]\label{thm_tri4to7}
Planar graphs without triangles adjacent to cycles of length from 4 to 7 are 3-colorable.
\end{theorem}

In this paper, we show that $(5,8)\in \cal{B}$, which leaves four pairs of integers $(5,6),(5,9),(7,8),$
$(8,9)$ unconfirmed as elements of $\cal{B}$. 

Recently, Mondal gave a proof of the result $(5,8)\in \cal{B}$ in \cite{Mondal_2011}.
Here we exhibit two couterexamples to the theorem proved in that paper which yields the result $(5,8)\in \cal{B}$.
We restated this theorem as follows.
Let $C$ be a cycle of length at most 12 in a plane graph without 4-, 5- and 8-cycles. $C$ is bad if it is of length 9 or 12 and the subgraph inside $C$ has a partition into 3- and 6-cycles; otherwise, $C$ is good.

\begin{theorem}[Theorem 2 in \cite{Mondal_2011}]\label{thm_wrong}
	Let $G$ be a graph without 4-, 5-, and 8-cycles. If $D$ is a good cycle of $G$, then every proper 3-coloring of $D$ can be extended to a proper 3-coloring of the whole graph $G$.
\end{theorem}

\noindent {\bf Counterexamples to Theorem \ref{thm_wrong}.} A plane graph $G_1$ consisting of a cycle $C$ of length 12, say $C:=[v_1 \ldots v_{12}]$, and a vertex $u$ inside $C$ connected to all of $v_1,v_2,v_6$. 
The graph $G_1$ contradicts Theorem \ref{thm_wrong}, since any proper 3-coloring of $C$ where $v_1,v_2,v_6$ receive pairwise distinct colors can not be extended to $G_1$.
Also, a plane graph $G_2$ consisting of a cycle $C$ of length 12 and a triangle $T$ inside $C$, say $C:=[v_1 \ldots v_{12}]$ and $T:=[u_1u_2u_3]$, and three more edges $u_1v_1,u_2v_4,u_3v_7$.
The graph $G_2$ contradicts Theorem \ref{thm_wrong}, since any proper 3-coloring of $C$ where $v_1,v_4,v_7$ receive the same color can not be extended to $G_2$ (see Figure \ref{fig_counterexample}).

\begin{figure}[h]
	\centering
	\includegraphics[width=6cm]{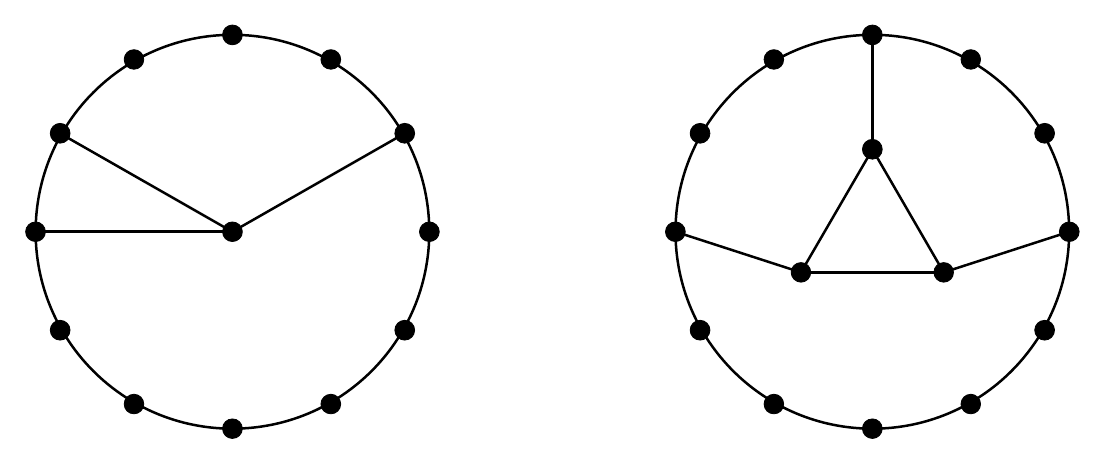}\\
	\caption{two graphs as counterexamples to Theorem \ref{thm_wrong}.}\label{fig_counterexample}
\end{figure}

\subsection{Notations and formulation of the main theorem}
The graphs considered in this paper are finite and simple. 
A graph is planar if it is embeddable into the Euclidean plane. A plane graph $(G,\Sigma)$ is a planar graph $G$ together with an embedding $\Sigma$ of $G$ into the Euclidean plane, that is, $(G,\Sigma)$ is a particular drawing of $G$ in the Euclidean plane.
In what follows, we will always say a plane graph $G$ instead of $(G,\Sigma)$, which causes no confusion since no two embeddings of the same graph $G$ will be involved in.  

Let $G$ be a plane graph and $C$ be a cycle of $G$.
By $Int(C)$ (or $Ext(C)$) we denote the subgraph of $G$ induced by the vertices lying inside (or outside) $C$.
The cycle $C$ is \textit{separating} if neither $Int(C)$ nor $Ext(C)$ is empty.
By $\overline{Int}(C)$ (or $\overline{Ext}(C)$) we denote the subgraph of $G$ consisting of $C$ and its interior (or exterior).
The cycle $C$ is \textit{triangular} if it is adjacent to a triangle, 
and $C$ is \textit{ext-triangular} if it is adjacent to a triangle of $\overline{Ext}(C)$.

The following theorem is the main result of this paper.
\begin{theorem} \label{thm45tri7}
Plane graphs with neither 4- and 5-cycles nor ext-triangular 7-cycles are 3-colorable.
\end{theorem}

As a consequence of Theorem \ref{thm45tri7}, the following corollary holds true.
\begin{corollary} \label{cor458}
Planar graphs without cycles of length 4, 5, 8 are 3-colorable, that is, $(5,8)\in \cal{B}$.
\end{corollary}
We remark that Theorem \ref{thm45tri7} implies the known result that $(5,7)\in \cal{B}$ as well.

Denote by $d(v)$ the degree of a vertex $v$, by $|P|$ the number of edges of a path $P$, by $|C|$ the length of a cycle $C$ and by $d(f)$ the size of a face $f$.  
A \textit{$k$-vertex} (or $k^+$-vertex, or $k^-$-vertex) is a vertex $v$ with $d(v)=k$ (or $d(v)\geq k$, or $d(v)\leq k$). 
Similar notations are used for paths, cycles, faces with $|P|,|C|,d(f)$ instead of $d(v)$, respectively.


Let $G[S]$ denote the subgraph of $G$ induced by $S$ with either $S\subseteq V(G)$ or $S\subseteq E(G)$.
A \textit{chord} of $C$ is an edge of $\overline{Int}(C)$ that connects two nonconsecutive vertices on $C$.
If $Int(C)$ has a vertex $v$ with three neighbors $v_1,v_2,v_3$ on $C$, then $G[\{vv_1,vv_2, vv_3\}]$ is called a \textit{claw} of $C$.
If $Int(C)$ has two adjacent vertices $u$ and $v$ such that $u$ has two neighbors $u_1,u_2$ on $C$ and $v$ has two neighbors $v_1,v_2$ on $C$, then $G[\{uv,uu_1,uu_2,vv_1,vv_2\}]$ is called a \textit{biclaw} of $C$.
If $Int(C)$ has three pairwise adjacent vertices $u,v,w$ which has a neighbor $u',v',w'$ on $C$ respectively, then $G[\{uv,vw,uw,uu',vv',ww'\}]$ is called a \textit{triclaw} of $C$.
If $G$ has four vertices $x,u,v,w$ inside $C$ and four vertices $x_1,x_2,v_1,w_1$ on $C$ such that $S=\{uv,vw,wu,ux,xx_1,xx_2,vv_1,ww_1\}\subseteq E(G)$, then $G[S]$ is called a \textit{combclaw} of $C$ (see Figure \ref{fig_claw}).
\begin{figure}[h]
  \centering
  \includegraphics[width=17cm]{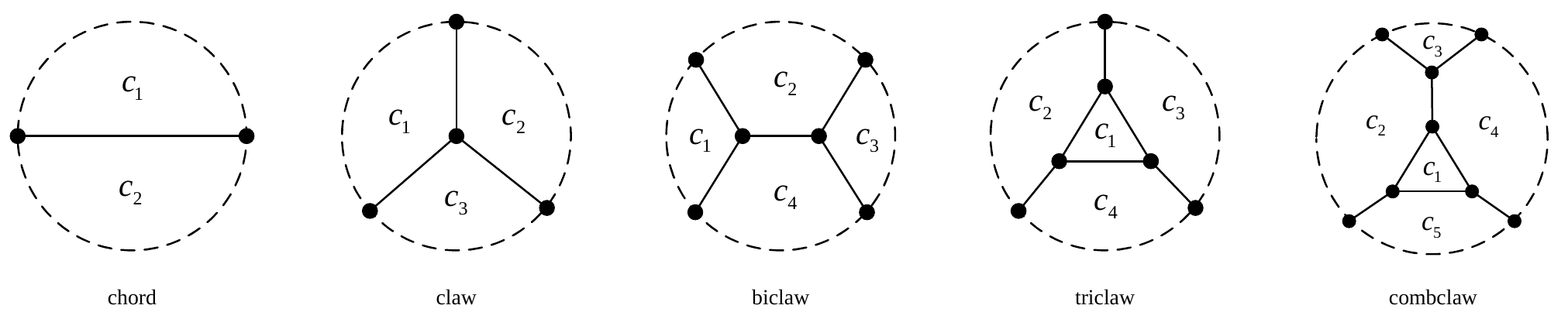}\\
  \caption{chord, claw, biclaw, triclaw and combclaw of a cycle}\label{fig_claw}
\end{figure}

A \textit{good cycle} is an 11$^-$-cycle that has none of claws, biclaws, triclaws and combclaws.
A \textit{bad cycle} is an 11$^-$-cycle that is not good.

Instead of Theorem \ref{thm45tri7}, it is easier for us to prove the following stronger one:
\begin{theorem} \label{thm45tri7ext}
Let $G$ be a connected plane graph with neither 4- and 5-cycles nor ext-triangular 7-cycles.
If $D$, the boundary of the exterior face of $G$, is a good cycle, then every proper 3-coloring of $G[V(D)]$ can be extended to a proper 3-coloring of $G$.
\end{theorem}

The proof of Theorem \ref{thm45tri7ext} will be proceeded by using discharging method and is given in the next section.
For more information on the discharging method, we refer readers to \cite{West2013}.
The rest of this section contributes to other needed notations.

Let $C$ be a cycle and $T$ be one of chords, claws, biclaws, triclaws and combclaws of $C$. We call the graph $H$ consisting of $C$ and $T$ a \textit{bad partition} of $C$. The boundary of any one of the parts, into which $C$ is divided by $H$, is called a \textit{cell} of $H$. Clearly, every cell is a cycle.
In case of confusion, let us always order the cells $c_1,\cdots,c_t$ of $H$ in the way as shown in Figure \ref{fig_claw}.
Let $k_i$ be the length of $c_i$. Then $T$ is further called a $(k_1,k_2)$-chord, a $(k_1,k_2,k_3)$-claw, a $(k_1,k_2,k_3,k_4)$-biclaw, a $(k_1,k_2,k_3,k_4)$-triclaw and a $(k_1,k_2,k_3,k_4,k_5)$-combclaw, respectively.


A vertex is \textit{external} if it lies on the exterior face; \textit{internal} otherwise.
A vertex (or an edge) is triangular if it is incident with a triangle.
We say a vertex is \textit{bad} if it is an internal triangular 3-vertex; \textit{good} otherwise.
A path is a \textit{splitting path} of a cycle $C$ if it has the two end-vertices on $C$ and all other vertices inside $C$.
A $k$-cycle with vertices $v_1,\ldots,v_k$ in cyclic order is denoted by $[v_1\ldots v_k]$.

Let $uvw$ be a path on the boundary of a face $f$ of $G$ with $v$ internal.
The vertex $v$ is $f$-heavy if both $uv$ and $vw$ are triangular and $d(v)\geq 5$, and is $f$-Mlight if both $uv$ and $vw$ are triangular and $d(v)=4$, and $f$-Vlight if neither $uv$ nor $vw$ is triangular and $v$ is triangular and of degree 4. A vertex is $f$-light if it is either $f$-Mlight or $f$-Vlight.

Denote by $\cal{G}$ the class of connected plane graphs with neither 4- and 5-cycles nor ext-triangular 7-cycles.

\section{The proof of Theorem \ref{thm45tri7ext}}
Suppose to the contrary that Theorem \ref{thm45tri7ext} is false.
From now on, let $G$ be a counterexample to Theorem \ref{thm45tri7ext} with fewest vertices.
Thus, we may assume that the boundary $D$ of the exterior face of $G$ is a good cycle, and there exists a proper 3-coloring $\phi$ of $G[V(D)]$ which cannot be extended to a proper 3-coloring of $G$.
By the minimality of $G$, we deduce that $D$ has no chord.
\subsection{Structural properties of the minimal counterexample $G$}
\begin{lemma} \label{lem_min degree}
Every internal vertex of $G$ has degree at least 3.
\end{lemma}

\begin{proof}
 Suppose to the contrary that $G$ has an internal vertex $v$ with $d(v)\leq 2$. We can extend $\phi$ to $G-v$ by the minimality of $G$, and then to $G$ by coloring $v$ different from its neighbors.
\end{proof}

\begin{lemma}
$G$ is 2-connected and therefore, the boundary of each face of $G$ is a cycle.
\end{lemma}

\begin{proof}
Otherwise, we can assume that $G$ has a pendant block $B$ with cut vertex $v$ such that $B-v$ does not intersect with $D$.
We first extend $\phi$ to $G-(B-v)$, and then 3-color $B$ so that the color assigned to $v$ is unchanged.
\end{proof}

\begin{lemma} \label{lem_sep good cycle}
$G$ has no separating good cycle.
\end{lemma}

\begin{proof}
Suppose to the contrary that $G$ has a separating good cycle $C$. We extend $\phi$ to $G-Int(C)$.
Furthermore, since $C$ is a good cycle, the color of $C$ can be extended to its interior.
\end{proof}


By the definition of a bad cycle, one can easily conclude the lemma as follows.

\begin{lemma} \label{bad cycle of H}
If $C$ is a bad cycle of a graph in $\cal{G}$, then $C$ has length either 9 or 11. Furthermore, if $|C|=9$, then $C$ has a (3,6,6)-claw or a (3,6,6,6)-triclaw; if $|C|=11$, then $C$ has a (3,6,8)-claw, or a (3,6,6,6)- or (6,3,6,6)-biclaw, or a (3,6,6,8)-triclaw, or a (3,6,6,6,6)-combclaw.
\end{lemma}

Notice that all 3- and 6- and 8-cycles of $G$ are facial, thus the following statement is a consequence of the previous lemma together with the fact that $G\in \cal{G}$.

\begin{lemma} \label{bad cycle of G}
$G$ has neither bad cycle with a chord nor ext-triangular bad 9-cycle.
\end{lemma}

\begin{lemma}\label{lem_splitting path}
Let $P$ be a splitting path of $D$ which divides $D$ into two cycles $D'$ and $D''$.
The following four statements hold true.
\begin{enumerate}[(1)]
  \item If $|P|=2$, then there is a triangle between $D'$ and $D''$.
  \item $|P|\neq 3$.
  \item If $|P|=4$, then there is a 6- or 7-cycle between $D'$ and $D''$.
  \item If $|P|=5$, then there is a $9^-$-cycle between $D'$ and $D''$.
\end{enumerate}
\end{lemma}

\begin{proof}
Since $D$ has length at most 11, we have $|D'|+|D''|=|D|+2|P|\leq 11+2|P|$.

(1) ~Let $P=xyz$. Suppose to the contrary that $|D'|, |D''| \geq 6$.
By Lemma \ref{lem_min degree}, $y$ has a neighbor other than $x$ and $z$, say $y'$. It follows that $y'$ is internal since otherwise $D$ is a bad cycle with a claw. Without loss of generality, let $y'$ lie inside $D'$. Now $D'$ is a separating cycle. By Lemma \ref{lem_sep good cycle}, $D'$ is not good, i.e., either $D'$ is bad or $|D'|\geq 12$.
Since every bad cycle has length either 9 or 11 by Lemma \ref{bad cycle of H},  we have $|D'|\geq 9$. Recall that $|D'|+|D''|\leq 15$, thus $|D'|=9$ and $|D''|=6$. Now $D'$ has either a (3,6,6)-claw or a (3,6,6,6)-triclaw by Lemma \ref{bad cycle of H}, which implies that $D$ has a biclaw or a combclaw respectively, a contradiction.

(2) Suppose to the contrary that $|P|=3$. Let $P=wxyz$. Clearly $|D'|, |D''| \geq 6$.
Let $x'$ and $y'$ be a neighbor of $x$ and $y$ not on $P$, respectively.
If both $x'$ and $y'$ are external, then $D$ has a biclaw.
Hence, we may assume $x'$ lies inside $D'$.
By Lemmas \ref{lem_sep good cycle} and \ref{bad cycle of H} and the inequality $|D'|+|D''|\leq 17$, we deduce that $D'$ is a bad cycle and $D''$ is a good $8^-$-cycle.
If $y'$ is internal, then $y'$ lies inside $D'$.
It follows with the specific interior of a bad cycle that $x'=y'$ and $D'$ has either a claw or a biclaw, which implies that $D$ has either a triclaw or a combclaw respectively, a contradiction.
Hence, $y'$ is external.
Since every bad cycle as well as every $6^-$- or 8-cycle contains no chord by Lemma \ref{bad cycle of G},
we deduce that $yy'$ is a (3,6)-chord of $D''$. It follows that $D'$ is a bad and ext-triangular 9-cycle, contradicting Lemma \ref{bad cycle of G}.

(3) ~Let $P=vwxyz$. Suppose to the contrary that $|D'|, |D''| \geq 8$.
Since $|D'|+|D''|\leq 19$, we have $|D'|, |D''| \leq 11$.
Since $G$ has no 4- and 5-cycles, if $G$ has an edge $e$ connecting two nonconsecutive vertices on $P$, then the cycle formed by $e$ and $P$ has to be a triangle, yielding a splitting 3-path of $D$, contradicting the statement (2).
Therefore, no pair of nonconsecutive vertices on $P$ are adjacent.

Let $w', x', y'$ be a neighbor of $w, x, y$ not on $P$, respectively.
The statement (2) implies that $x'$ is internal.
Let $x'$ lie inside of $D'$.
Thus $D'$ is a bad 9- or 11-cycle.
If $D'$ is a bad 11-cycle, then $D''$ is a facial 8-cycle, and thus both $w'$ and $y'$ lie in $\overline{Int}(D')$, which is impossible by the interior of a bad cycle.
Hence, $D'$ is a bad 9-cycle.
By the statement (1), if $w'\in V(D'')$, then $G$ has the triangle $[vww']$, which makes $D'$ ext-triangular, a contradiction.
Hence, $w'\notin V(D'')$.
Furthermore, as a bad cycle, $D'$ has no chord by Lemma \ref{bad cycle of G}, thus $w'$ is internal.
If $w'$ lies inside $D'$, then it gives the interior of $D'$ no other choices but $w'=x'$ and $D'$ has a $(3,6,6)$-claw, in which case this claw contains a splitting 3-path of $D$, a contradiction. Hence, $w'$ lies inside $D''$. Similarly, we can deduce that $y'$ lies inside $D''$ as well.
Note that $|D''|\in \{8,9,10\}$, thus $D''$ is a bad 9-cycle but has to contain both $w'$ and $y'$ inside, which is impossible.

(4) ~Let $P=uvwxyz$. Suppose to the contrary that $|D'|, |D''| \geq 10$.
Since $|D'|+|D''|\leq 21$, we have $|D'|, |D''| \leq 11$.
By similar argument as in the proof of the statement (3), one can conclude that $G$ has no edge connecting two nonconsecutive vertices on $P$.
Let $v', w', x', y'$ be a neighbor of $v, w, x, y$ not on $P$, respectively.

The statement (2) implies that both $w'$ and $x'$ are internal.
Let $w'$ lie inside $D'$. It follows that $D'$ is a bad 11-cycle and $D''$ is a 10-cycle.
Thus $x'$ also lies inside $D'$ and furthermore, $x'=w'$ and $D'$ is a bad cycle with either a (3,6,8)-claw or a (3,6,6,6)-biclaw.
It follows that $v', y'\in V(D'')$.
By the statement (1), $G$ has two triangles $[uvv']$ and $[yy'z]$, at least one of them is adjacent to a 7-cycle of $\overline{Int}(D')$, a contradiction.
\end{proof}

\begin{lemma}\label{pro_operation}
Let $G'$ be a connected plane graph obtained from $G$ by deleting a set of internal vertices and identifying two other vertices so that at most one pair of edges are merged.
If we
\begin{enumerate}[($a$)]
  \item identify no two vertices of $D$, and create no edge connecting two vertices of $D$, and
  \item create no $6^-$-cycle and ext-triangular 7-cycle,
\end{enumerate}
then $\phi$ can be extended to $G'$.
\end{lemma}

\begin{proof}
The item $(a)$ guarantees that $D$ is unchanged and bounds $G'$, and $\phi$ is a proper 3-coloring of $G'[V(D)]$.
By item $(b)$, the graph $G'$ is simple and $G'\in \cal{G}$.
Hence, to extend $\phi$ to $G'$ by the minimality of $G$, it remains to show that $D$ is a good cycle of $G'$.

Suppose to the contrary that $D$ has a bad partition $H$ in $G'$.
Clearly, $H$ has a 6-cell $C'$ such that the intersection between $D$ and $C'$ is a path $v_1\ldots v_k$ of length $k-1$ with $k\in \{4,5\}$.
Since we create no 6-cycles, $C'$ corresponds to a 6-cycle $C$ of the original graph $G$.
Recall that at most one pair of edges are merged during the process from $G$ to $G'$,
we deduce that the intersection between $D$ and $C$ is a path $P$ of one of the forms $v_1\ldots v_k,v_1\ldots v_{k-1},v_2\ldots v_k$.
Thus, $|P|\in \{3,4,5\}$.
If $|P|\in\{4,5\}$, then $C$ contains a splitting 3- or 2-path of $D$ in $G$, yielding a contradiction by Lemma \ref{lem_splitting path}.
Hence, $|P|=3$ and so $k=4$. By the choice of the 6-cell $C'$, we may assume that the bad partition $H$ has either a (3,6,6,6)- or (3,6,6,8)-triclaw.
Now $H$ contains three splitting 3-paths of $D$, at least one of them does not contain the identified vertex of $G'$ no matter where it is, yielding the existence of a splitting 3-path of $D$ in $G$, contradicting Lemma \ref{lem_splitting path}.
\end{proof}

\begin{lemma} \label{lem_6FaceWithout33}
$G$ has no edge $uv$ incident with a 6-face and a 3-face such that both $u$ and $v$ are internal 3-vertices and therefore, every bad cycle of $G$ has either a (3,6,6)- or (3,6,8)-claw or a (3,6,6,6)-biclaw.
\end{lemma}

\begin{proof}
Suppose to the contrary that such an edge $uv$ exists.
Denote by $[uvwxyz]$ and $[uvt]$ the 6-face and 3-face, respectively.
Lemma \ref{lem_splitting path} implies that not both of $w$ and $z$ are external vertices.
Without loss of generality, we may assume that $w$ is internal.
Let $G'$ be the graph obtained from $G$ by deleting $u$ and $v$, and identifying $w$ with $y$ so that $wx$ and $yx$ are merged.
Clearly, $G'$ is a plane graph on fewer vertices than $G$.
We will show that both the items in Lemma \ref{pro_operation} are satisfied.

Since $w$ is internal, we identify no two vertices on $D$.
If we create an edge connecting two vertices on $D$, then $w$ has a neighbor $w_1$ not adjacent to $y$ and both $y$ and $w_1$ are external.
But now, Lemma \ref{lem_splitting path} implies that $x$ is external and thus, $[ww_1x]$ is a triangle which makes the 7-cycle $[utvwxyz]$ ext-triangular. Hence, the item $(a)$ holds.

Suppose we create a $6^-$-cycle or an ext-triangular 7-cycle $C'$.
Thus $G$ has a $7^-$-path $P$ between $w$ and $y$ corresponding to $C'$.
If $x\in V(P)$, then neither $wx$ nor $xy$ are on $P$ since otherwise, $C'$ already exists in $G$. Hence, the paths $wxy$ and $P$ form two cycles, both of them has length at least 6.  It follows that $|P|\geq 10$, a contradiction.
Hence, we may assume that $x\notin V(P)$. The paths $P$ and $wxy$ form a $9^-$-cycle, say $C$.
By Lemma \ref{lem_min degree}, we may let $x_1$ be a neighbor of $x$ other than $y$ and $w$.
We have $x_1\notin V(P)$, since otherwise $P$ has length at least 8.
Now $C$ has to contain either $u$ and $v$ or $x_1$ inside, which implies that $C$ is a bad 9-cycle.
By Lemma \ref{bad cycle of G}, $C$ is not ext-triangular. Thus $C'$ is a 7-cycle that is not ext-triangular, contradicting the supposition.
Hence, the item $(b)$ holds.

By Lemma \ref{pro_operation}, the pre-coloring $\phi$ can be extended to $G'$.
Since $z$ and $w$ receive different colors, we can properly color $v$ and $u$, extending $\phi$ further to $G$.
\end{proof}

We follow the notations of $M$-face and $MM$-face in \cite{BorodinEtc2005303}, and define weak tetrads.
An $M$-face is an 8-face $f$ containing no external vertices with boundary $[v_1\ldots v_8]$ such that the vertices $v_1,v_2,v_3,v_5,v_6,v_7$ are of degree 3 and the edges $v_1v_2,v_3v_4,v_4v_5,v_6v_7$ are triangular.
An $MM$-face is an 8-face $f$ containing no external vertices with boundary $[v_1\ldots v_8]$ such that
$v_2$ and $v_7$ are of degree 4 and other six vertices on $f$ are of degree 3, and the edges $v_1v_2,v_2v_3,v_4v_5,v_6v_7,v_7v_8$ are triangular.
A weak tetrad is a path $v_1\ldots v_5$ on the boundary of a face $f$ such that
both the edges $v_1v_2$ and $v_3v_4$ are triangular, all of $v_1,v_2,v_3,v_4$ are internal 3-vertices,
and $v_5$ is either of degree 3 or $f$-light.

\begin{lemma}\label{lem_tetrad}
$G$ has no weak tetrad and therefore, every face of $G$ contains no five consecutive bad vertices.
\end{lemma}

\begin{proof}
Suppose to the contrary that $G$ has a weak tetrad $T$ following the notation used in the definition.
Denote by $v_0$ the neighbor of $v_1$ on $f$ with $v_0\neq v_2$.
Denote by $x$ the common neighbor of $v_1$ and $v_2$, and $y$ the common neighbor of $v_3$ and $v_4$.
If $x=v_0$, then $v_1$ is an internal 2-vertex, contradicting Lemma \ref{lem_min degree}.
Hence, $x\neq v_0$ and similarly, $x\neq v_3$.
Since $G$ has no 4- or 5-cycles, $x\notin \{v_4,v_5\}$.
Concluding above, $x\notin v_0\cup V(T)$. Similarly, $y\notin v_0\cup V(T)$.
Moreover, $x\neq y$ since otherwise $[v_1v_2v_3x]$ is a 4-cycle.
We delete $v_1,\ldots,v_4$ and identity $v_0$ with $y$, obtaining a
plane graph $G'$ on fewer vertices than $G$.
We will show that both the items in Lemma \ref{pro_operation} are satisfied.

Suppose that we create a $6^-$-cycle or an ext-triangular 7-cycle $C'$.
Thus $G$ has a $7^-$-path $P$ between $v_0$ and $y$ corresponding to $C'$.
If $x\in V(P)$, then the cycle formed by $P$ and $v_0v_1x$ has length at least 6 and the one formed by $P$ and $xv_2v_3y$ has length at least 8, which gives $|P|\geq 9$, a contradiction.
Hence, $x\notin V(P)$. The paths $P$ and $v_0v_1v_2v_3y$ form a $11^-$-cycle, say $C$.
Now $C$ contains either $x$ or $v_4$ inside.
Thus, $C$ is a bad cycle.
By Lemma \ref{lem_6FaceWithout33}, $C$ has either a (3,6,6)- or (3,6,8)-claw or a (3,6,6,6)-biclaw.
Note that both the two faces incident with $v_2v_3$ has length at least 8,
thus $C$ has a bad partition owning an 8-cell no matter which one of $x$ and $v_4$ lies inside $C$.
It follows that $C$ has a (3,6,8)-claw.
If $x$ lies inside $C$, then the 6-cell is adjacent to the triangle $[xv_1v_2]$ with $d(v_1)=d(x)=3$, contradicting Lemma \ref{lem_6FaceWithout33}.
Hence, $v_4$ lies inside $C$.
Note that $v_4v_5$ is incident with the 6-cell and the 8-cell, we deduce that $v_5$ is not $f$-light.
By the assumption of $T$ as a weak tetrad, we may assume that $d(v_5)=3$.
We delete $v_5$ together with other vertices of $T$ and repeat the argument above, yielding a contradiction.
Therefore, the item $(b)$ holds.

Suppose we identify two vertices on $D$ or create an edge connecting two vertices on $D$.
Thus there is a splitting 4- or 5-path $Q$ of $D$ containing the path $v_0v_1v_2v_3y$.
By Lemma \ref{lem_splitting path}, $Q$ together with $D$ forms a $9^-$-cycle which corresponds to a $5^-$-cycle in $G'$.
Since we create no $6^-$-cycle, a contradiction follows. Hence, the item $(a)$ holds.

By Lemma \ref{pro_operation}, the pre-coloring $\phi$ can be extended to $G'$.
We first properly color $v_5$ (if needed), $v_4,v_3$ in turn.
Since $v_0$ and $v_3$ receive different colors, we can properly color $v_1$ and $v_2$, extending $\phi$ further to $G$.
\end{proof}

\begin{lemma}\label{lem_M-face}
$G$ has no $M$-face.
\end{lemma}

\begin{proof}
Suppose to the contrary that $G$ has an $M$-face $f$ following the notation used in the definition.
For $(i,j)\in\{(1,2),(3,4),(4,5),(6,7)\}$, denote by $t_{ij}$ the common neighbor of $v_i$ and $v_j$.
By similar argument as in the proof of previous lemma, we deduce that
the vertices $t_{12},t_{34},t_{45},t_{67}$ are pairwise distinct and not incident with $f$.
We delete $v_1,v_2,v_3,v_5,v_6,v_7$ and identity $v_4$ with $v_8$, obtaining a
plane graph $G'$ on fewer vertices than $G$.
We will show that both the items in Lemma \ref{pro_operation} are satisfied.

Suppose that we create a $6^-$-cycle or an ext-triangular 7-cycle $C'$.
Thus $G$ has a $7^-$-path $P$ between $v_4$ and $v_8$ corresponding to $C'$.
By the symmetry of an $M$-face, we may assume that $P$ together with the path $v_4\ldots v_8$ forms a $11^-$-cycle $C$ containing $v_1,v_2,v_3$ inside.
It follows with Lemma \ref{lem_6FaceWithout33} that $C$ is a bad cycle with a $(3,6,6,6)$-claw.
But now $\overline{Int}(C)$ contains $f$ that is an 8-face, a contradiction. Therefore, the item $(b)$ holds.

The satisfaction of the item $(a)$ can be proved in a similar way as in the proof of previous lemma.

By Lemma \ref{pro_operation}, the pre-coloring $\phi$ can be extended to $G'$.
Since we first color $v_3$ different from $v_8$, both $v_1$ and $v_2$ can be properly colored.
Finally, color $v_5,v_6,v_7$ in the same way, extending $\phi$ further to $G$.
\end{proof}

\begin{lemma}\label{lem_MM-face}
$G$ has no $MM$-face.
\end{lemma}

\begin{proof}
Suppose to the contrary that $G$ has an $MM$-face $f$ following the notation used in the definition.
For $(i,j)\in\{(1,2),(2,3),(4,5),(6,7),(7,8)\}$, denote by $t_{ij}$ the common neighbor of $v_i$ and $v_j$.
Similarly, we deduce that the vertices $t_{12},t_{23},t_{45},t_{67},t_{78}$ are pairwise distinct and not incident with $f$.
We delete all the vertices of $f$ and identity $t_{12}$ with $t_{67}$, obtaining a
plane graph $G'$ on fewer vertices than $G$.
To extend $\phi$ to $G'$, it suffices to fulfill the item $(a)$ of Lemma \ref{pro_operation}, as what we did in previous lemma.

Suppose that we create a $6^-$-cycle or an ext-triangular 7-cycle $C'$.
Thus $G$ has a $7^-$-path $P$ between $t_{12}$ and $t_{67}$ corresponding to $C'$.
If $t_{78}\in V(P)$, then both the cycles formed by $P$ and $t_{12}v_1v_8t_{78}$ and by $P$ and $t_{78}v_7t_{67}$ have length at least 8, which gives $|P|\geq 11$, a contradiction.
Hence, $t_{78}\notin V(P)$. The paths $P$ and $t_{12}v_1v_8v_7t_{67}$ form a $11^-$-cycle, say $C$.
It follows that $C$ is a bad cycle containing either $t_{78}$ or $v_2,\ldots,v_6$ inside, that is, either $C$ has a bad partition owning two $8^+$-cell or $C$ contains five vertices inside, a contradiction in any case.

We further extend $\phi$ from $G'$ to $G$ as follows.
Let $\alpha,\beta$ and $\gamma$ be the three colors used in $\phi$.
First regardless the edge $v_1v_8$, we can properly color $v_2,v_1,v_3$ and $v_7,v_8,v_6$.
If $v_1$ and $v_8$ receive different colors and so do $v_3$ and $v_6$, then $v_4$ and $v_5$ can be properly colored, we are done.
Hence, we may assume without loss of generality that $v_1$ and $v_8$ receive the same color, say $\beta$.
Let $\alpha$ be the color assigned to $t_{12}$ and $t_{67}$. Thus $v_2$ and $v_7$ are colored with $\gamma$ and $t_{78}$ is colored with $\alpha$.
We recolor $v_8,v_7,v_6$ with $\gamma, \beta, \gamma$ respectively.
Now $v_1$ and $v_8$ receive different colors and so do $v_3$ and $v_6$.
Again $v_4$ and $v_5$ can be properly colored, we are also done.
\end{proof}

\subsection{Discharging in $G$}
\label{secch}
Let $V=V(G)$, $E=E(G)$, and $F$ be the set of faces of $G$.
Denote by $f_0$ the exterior face of $G$.
Give initial charge $ch(x)$ to each element $x$ of $V\cup F$, where $ch(f_0)=d(f_0)+4$, $ch(v)=d(v)-4$ for $v\in V$, and $ch(f)=d(f)-4$ for $f\in F\setminus \{f_0\}$.
Discharge the elements of $V\cup F$ according to the following rules:
\begin{enumerate}[$R1.$]
  \item Every internal 3-face receives $\frac{1}{3}$ from each incident vertex.
  \item Every internal $6^+$-face sends $\frac{2}{3}$ to each incident 2-vertex.
  \item Every internal $6^+$-face sends each incident 3-vertex $v$ charge $\frac{2}{3}$ if $v$ is triangular, and charge $\frac{1}{3}$ otherwise.
  \item Every internal $6^+$-face $f$ sends $\frac{1}{3}$ to each $f$-light vertex, and receives $\frac{1}{3}$ from each $f$-heavy vertex.
  \item Every internal $6^+$-face receives $\frac{1}{3}$ from each incident external $4^+$-vertex.
  \item The exterior face $f_0$ sends $\frac{4}{3}$ to each incident vertex.
\end{enumerate}

Let $ch^*(x)$ denote the final charge of each element $x$ of $V\cup F$ after discharging.
On one hand, by Euler's formula we deduce $\sum\limits_{x\in V\cup F}ch(x)=0.$
Since the sum of charges over all elements of $V\cup F$ is unchanged, it follows that $\sum\limits_{x\in V\cup F}ch^*(x)=0.$ On the other hand, we show that $ch^*(x)\geq 0$ for $x\in V\cup F\setminus \{f_0\}$ and  $ch^*(f_0)> 0$. Hence, this obvious contradiction completes the proof of Theorem \ref{thm45tri7ext}.
It remains to show that $ch^*(x)\geq 0$ for $x\in V\cup F\setminus \{f_0\}$ and  $ch^*(f_0)> 0$.

We remark that the discharging rules can be tracked back to the one used in \cite{BorodinEtc2005303}.

\begin{lemma}
$ch^*(v)\geq0$ for $v\in V$.
\end{lemma}
\begin{proof}
First suppose that $v$ is external. Since $D$ is a cycle, $d(v)\geq 2$.
If $d(v)=2$, then since $D$ has no chord, the internal face incident with $v$ is not a triangle and sends $\frac{2}{3}$ to $v$ by $R2$. Moreover, $v$ receives $\frac{4}{3}$ from $f_0$ by $R6$, which gives  $ch^*(v)=d(v)-4+\frac{2}{3}+\frac{4}{3}=0$.
If $d(v)=3$, then $v$ sends charge to at most one 3-face by $R1$ and thus $ch^*(v)\geq d(v)-4-\frac{1}{3}+\frac{4}{3}=0$.
If $d(v)\geq 4$, then $v$ sends at most $\frac{1}{3}$ to each incident internal face by $R1$ and $R5$, yielding $ch^*(v)\geq d(v)-4-\frac{1}{3}(d(v)-1)+\frac{4}{3}>0$.
Hence, we are done in any case.

It remains to suppose that $v$ is internal. By Lemma \ref{lem_min degree}, $d(v)\geq 3.$
If $d(v)=3$, then we have $ch^*(v)=d(v)-4-\frac{1}{3}+\frac{2}{3}\times 2=0$ by $R1$ and $R3$ when $v$ is triangular,
and $ch^*(v)=d(v)-4+\frac{1}{3}\times 3=0$ by $R3$ when $v$ not.
If $d(v)=4$, then $v$ is incident with $k$ 3-faces with $k\leq 2$. By $R1$ and $R4$, we have $ch^*(v)=d(v)-4-\frac{1}{3}\times 2+\frac{1}{3}\times 2=0$ when $k=2$, $ch^*(v)=d(v)-4-\frac{1}{3}+\frac{1}{3}=0$ when $k=1$, and $ch^*(v)=d(v)-4=0$ when $k=0$.
If $d(v)=5$, then $v$ sends charge to at most two 3-faces by $R1$ and to at most one $6^+$-face by $R4$, which gives $ch^*(v)\geq d(v)-4-\frac{1}{3}\times 2-\frac{1}{3}=0$.
Hence, we may next assume that $d(v)\geq 6.$
Since $v$ sends at most $\frac{1}{3}$ to each incident face by our rules,
we get $ch^*(v)\geq d(v)-4-\frac{1}{3}d(v)\geq 0.$
\end{proof}

\begin{lemma}
$ch^*(f_0)>0$.
\end{lemma}
\begin{proof}
Recall that $ch(f_0)=d(f_0)+4$ and $d(f_0)\leq 11$. We have $ch^*(f_0)\geq d(f_0)+4-\frac{4}{3}d(f_0)> 0$ by $R6$.
\end{proof}

\begin{lemma}
$ch^*(f)\geq 0$ for $f\in F\setminus \{f_0\}$.
\end{lemma}
\begin{proof}
We distinguish cases according to the size of $f$. Since $G$ has no 4- and 5-cycle, $d(f)\notin \{4,5\}$.

If $d(f)=3$, then $f$ receives $\frac{1}{3}$ from each incident vertices by $R1$, which gives $ch^*(f)=d(f)-4+\frac{1}{3}\times 3=0.$

Let $d(f)=6$. For any incident vertex $v$, by the rules,
$f$ sends to $v$ charge $\frac{2}{3}$ if $v$ is either of degree 2 or bad, and charge at most $\frac{1}{3}$ otherwise.
Since $G$ has no ext-triangular 7-cycles, $f$ is adjacent to at most one 3-face.
Furthermore, by Lemma \ref{lem_6FaceWithout33}, $f$ contains at most one bad vertex.
If $f$ contains a 2-vertex, say $u$, we can deduce with Lemma \ref{lem_splitting path} that
$u$ is the unique 2-vertex of $f$ and the two neighbors of $u$ on $f$ are external $3^+$-vertices which receive nothing from $f$.
It follows that $ch^*(f)\geq d(f)-4-\frac{2}{3}-\frac{2}{3}-\frac{1}{3}\times 2=0.$
Hence, we may assume that $f$ contains no 2-vertices.
If $f$ has no bad vertices, then $f$ sends each incident vertex at most $\frac{1}{3}$, which gives $ch^*(f)\geq d(f)-4-\frac{1}{3}d(f)=0$.
Hence, we may let $x$ be a bad vertex of $f$. Denote by $y$ the other common vertex between $f$ and the triangle adjacent to $f$.
By Lemma \ref{lem_6FaceWithout33} again, $y$ is not a bad vertex, i.e., $y$ is either an internal $4^+$-vertex or an external $3^+$-vertex.
By our rules, $f$ sends nothing to $y$, yielding $ch^*(f)\geq d(f)-4-\frac{2}{3}-\frac{1}{3}\times 4=0.$

Let $d(f)=7$. Since $G$ has no ext-triangular 7-cycles, $f$ contains no bad vertices.
Moreover, by Lemma \ref{lem_splitting path}, we deduce that $f$ has at most two 2-vertices.
Thus, $ch^*(f)\geq d(f)-4-\frac{2}{3}\times 2-\frac{1}{3}\times 5=0.$

Let $d(f)\geq 8$. On the hand, if $f$ contains precisely one external vertex, say $w$, then
$d(w)\geq 4$ and so $f$ receives $\frac{1}{3}$ from $w$ by $R5$.
Furthermore, since $f$ contains no weak tetrad by Lemma \ref{lem_tetrad}, $f$ has a good vertex other than $w$ and sends at most $\frac{1}{3}$ to it.
Hence, $ch^*(f)\geq d(f)-4+\frac{1}{3}-\frac{1}{3}-\frac{2}{3}(d(f)-2)\geq 0.$
On the other hand, if $f$ contains at least two external vertices, then at least two of them are of degree more than 2.
Since $f$ sends nothing to external $3^+$-vertices, we have $ch^*(f)\geq d(f)-4-\frac{2}{3}(d(f)-2)\geq 0$.
By the two hands above, we may assume that all the vertices of $f$ are internal.
We distinguish two cases.

Case 1: assume that $d(f)=8$.
Denote by $r$ the number of bad vertices of $f$.
We have $ch^*(f)\geq d(f)-4-\frac{2}{3}r-\frac{1}{3}(d(f)-r)=\frac{4-r}{3}\geq 0$, provided by $r\leq 4$.
Since $f$ contains no weak tetrad, $r\leq 6$.
Hence, we may assume that $r\in\{5,6\}$.
For $r=5$, we claim that $f$ has a vertex failing to take charge from $f$, which gives $ch^*(f)\geq d(f)-4-\frac{2}{3}\times 5-\frac{1}{3}\times 2=0.$
Suppose to the contrary that no such vertex exists.
Thus, the bad vertices of $f$ can be paired so that any good vertex of the path of $f$ between each pair is $f$-Mlight, contradicting the parity of $r$.
For $r=6$, since again $f$ contains no five consecutive bad vertices, these six bad vertices of $f$ are divided by the two good ones into
cyclically either 3+3 or 2+4.
We may assume that $f$ has a good vertex that is either $f$-light or of degree 3, since otherwise we are done with $ch^*(f)\geq d(f)-4-\frac{2}{3}\times 6=0$. Denote by $u$ such a good vertex and by $v$ the other one.
By the drawing of $u$ and of the 3-faces adjacent to $f$,
we deduce that, for the case 3+3, $f$ is an $M$-face, contradicting Lemma \ref{lem_M-face},
and for the case 2+4, if $u$ is $f$-Mlight then either $f$ is an $MM$-face or $v$ is $f$-heavy; otherwise $f$ contains a weak tetrad.
It follows with Lemmas \ref{lem_MM-face} and \ref{lem_tetrad} that $v$ is $f$-heavy, which is the only possible case.
Hence, $f$ receives $\frac{1}{3}$ from $v$ by $R4$, yielding $ch^*(f)\geq ch(f)-4-\frac{2}{3}\times 6+\frac{1}{3}-\frac{1}{3}=0$.

Case 2: assume that $d(f)\geq 9$.
By Lemma \ref{lem_tetrad}, we deduce that $f$ contains at least two good vertices,
each of them receives at most $\frac{1}{3}$ from $f$.
Thus, $ch^*(f)\geq d(f)-4-\frac{2}{3}(d(f)-2)-\frac{1}{3}\times 2=\frac{d(f)-10}{3}\geq 0$, provided by $d(f)\geq 10$.
It remains to suppose $d(f)=9$. If $f$ has at most six bad vertices, then $ch^*(f)\geq d(f)-4-\frac{2}{3}\times 6-\frac{1}{3}\times 3=0.$
Hence, we may assume that $f$ has precisely seven bad vertices.
By the same argument as for the case $d(f)=8$ and $f$ has five bad vertices above, $f$ has a vertex failing to take charge from $f$, which gives $ch^*(f)\geq d(f)-4-\frac{2}{3}\times 7-\frac{1}{3}=0.$
\end{proof}

By the previous three lemmas, the proof of Theorem \ref{thm45tri7ext} is completed.

\section{Acknowledgement}

The first author is supported by Deutsche Forschungsgemeinschaft (DFG) grant STE 792/2-1.
The fourth author is supported by National Natural Science Foundation of China (NSFC) 11271335.

\end{document}